\documentclass[a4paper,reqno]{amsart}

\usepackage[american]{babel}
\usepackage[T1]{fontenc}
\usepackage[utf8]{inputenc}
\usepackage{amsmath,amssymb,amsfonts,amsthm,amstext}
\usepackage{dsfont,colonequals,enumerate,mathrsfs}
\usepackage{color}
\usepackage{microtype}
\usepackage[numbers,sort&compress]{natbib}

\newcommand{\dx}{\mathrm{d}}
\newcommand{\coloneqq}{\colonequals}
\newcommand{\apply}[3][]{\left<#2,#3\right>_{#1}}
\newcommand{\scalar}[3][]{\left(#2\mid#3\right)_{#1}}
\renewcommand{\phi}{\varphi}
\newcommand{\e}{\mathrm{e}}
\newcommand{\setone}{\mathds{1}}
\newcommand{\eps}{\varepsilon}
\renewcommand{\rho}{\varrho}

\DeclareMathOperator{\Div}{div}

\DeclareMathOperator{\Real}{Re}
\DeclareMathOperator{\sgn}{sgn}

\newenvironment{smallpmatrix}{\left(\begin{smallmatrix}}{\end{smallmatrix}\right)}

\theoremstyle{definition}
\newtheorem{theorem}{Theorem}[section]
\newtheorem{proposition}[theorem]{Proposition}

\newtheorem{lemma}[theorem]{Lemma}
\newtheorem{remark}[theorem]{Remark}
\newtheorem{example}[theorem]{Example}

\newtheorem{convention}[theorem]{Convention}
\newtheorem{notation}[theorem]{Notation}

\title[Elliptic and Parabolic Regularity on Lipschitz Domains]{Regularity of Solutions of Linear Second Order Elliptic and Parabolic Boundary Value Problems on Lipschitz Domains}
\author{Robin Nittka}
\address{Robin Nittka\\University of Ulm\\Institute of Applied Analysis\\89069 Ulm\\Germany}
\email{robin.nittka@uni-ulm.de}

\numberwithin{equation}{section}
\begin{document}
\begin{abstract}
	For a linear, strictly elliptic second order differential
	operator in divergence form with bounded, measurable
	coefficients on a Lipschitz domain $\Omega$
	we show that solutions of the corresponding elliptic problem
	with Robin and thus in
	particular with Neumann boundary conditions are H\"older
	continuous for sufficiently $L^p$-regular right-hand sides.
	From this we deduce that the parabolic
	problem with Robin or Wentzell-Robin boundary conditions
	are well-posed on $\mathrm{C}(\overline{\Omega})$.
\end{abstract}

\maketitle

\section{Introduction}
In this article we show that solutions of elliptic
Robin boundary value problems on a Lipschitz domain
$\Omega \subset \mathds{R}^N$ are
H\"older regular if the right-hand side is smooth
enough in an $L^p$-sense.
In particular, this result applies to Neumann boundary conditions.

More precisely, let $L$ be a strictly elliptic operator
in divergence form
\begin{equation}\label{eq:diff_operator}
	Lu = -\sum_{j=1}^N D_j \Bigl( \sum_{i=1}^N a_{ij} D_iu + b_j u \Bigr) + \sum_{i=1}^N c_i D_iu + du.
\end{equation}
with bounded, measurable coefficients.
We consider elliptic problems that formally
take the form
\begin{equation}\label{eq:formal_problem}
	\left\{ \begin{aligned}
		Lu & = f_0 - \sum_{j=1}^N D_jf_j && \text{on } \Omega, \\
		\frac{\partial u}{\partial \nu_L} + \beta u & = g + \sum_{j=1}^N f_j \nu_j && \text{on } \partial\Omega,
	\end{aligned} \right.
\end{equation}
where we set
\[
	\frac{\partial u}{\partial \nu_L}
		\coloneqq \sum_{j=1}^N \Bigl( \sum_{i=1}^N a_{ij} D_iu + b_j u \Bigr) \nu_j,
\]
and $\nu$ denotes the outer normal on $\partial\Omega$.
We assume $\beta$ to be bounded and measurable, but make no assumptions
on the sign of $\beta$.

In Section~\ref{sec:prelim} we explain what it meant by a
weak solution of~\eqref{eq:formal_problem}.
Section~\ref{sec:elliptic} is devoted to $L^p$-regularity
and H\"older regularity results for solutions of~\eqref{eq:formal_problem},
which are summarized in Theorem~\ref{thm:robin_reg}.
The main idea is to extend weak solutions by reflection at the
boundary, to show that this extension again solves an elliptic
problem, and then to apply interior regularity results
due to de Giorgi, Nash, and Moser.
This strategy is known, see for example~\cite[Section~2.4.3]{Troi87}
or~\cite[Remark~3.10]{BassHsu91},
but it seems that until now it has not been exploited to this extent.

In particular, it will follow that if $f_0 \in L^{p/2}(\Omega)$,
$f_j \in L^p(\Omega)$, $j=1,\dots,N$ and $g \in L^{p-1}(\partial\Omega)$
for $p > N$, then every solution $u$ of~\eqref{eq:formal_problem}
is H\"older continuous on $\Omega$.
Weaker versions of this result can be found in~\cite{BassHsu91,Warma06,Daners09,AR97}.
On the other hand, a stronger version of this result
has been obtained in~\cite{GR01}, but by considerably more difficult methods.
Thus I believe that the simplicity of this approach
has still its own appeal.

Using the elliptic regularity result we attack parabolic
problems in spaces of continuous functions in Section~\ref{sec:parabolic}.
More precisely, we consider the initial value problems
\begin{equation}\label{eq:robin_parabolic}
	\left\{\begin{aligned}
		\dot{u}(t,x) & = -Lu(t,x), && t > 0, \; x \in \Omega, \\
		\tfrac{\partial u}{\partial \nu_L}(t,z) + \beta u(t,z) & = 0, && t \ge 0, \; z \in \partial\Omega, \\
		u(0,x) & = u_0(x), && x \in \Omega,
	\end{aligned}\right.
\end{equation}
i.e., Robin or Neumann boundary conditions,
and
\begin{equation}\label{eq:wentzell_parabolic}
	\left\{\begin{aligned}
		\dot{u}(t,x) & = -Lu(t,x) && t > 0, x \in \Omega, \\
		-Lu(t,z) + \tfrac{\partial u(t,z)}{\partial \nu_L} + \beta u(t,z) & = 0, && t > 0, z \in \partial\Omega, \\
		u(0,x) & = u_0(x), && x \in \Omega,
	\end{aligned}\right.
\end{equation}
i.e., Wentzell-Robin boundary conditions.
These equations give rise to strongly continuous semigroups
in appropriate Hilbert spaces. These semigroups have extensively
been studied, see for example~\cite{AtE97,Daners00,Daners00_robin,AW03,AMPR03,CFGGR08}.
In special cases, it is known that the solution operators for these equations
define strongly continuous semigroups also in the space of continuous functions
on $\Omega$, see~\cite{BassHsu91,FT95,FGGR02,Engel03,Warma06}.
We extend these results to the case of arbitrary strongly elliptic
differential operators with bounded, measurable coefficients.

For simplicity we consider second order linear equations only.
We will work with bounded,
real-valued coefficients and pure Robin boundary
conditions, i.e., we do not allow for Dirichlet or mixed
boundary conditions.
We will not investigate whether the operators
generate semigroups on spaces of H\"older continuous
functions. In the generation results for Wentzell-Robin
boundary conditions we will in addition assume that
the first order coefficient of the elliptic operator
is Lipschitz continuous.

Probably the methods and main ideas of this article still apply without
the above restrictions, and I will attempt to
generalize the theorems accordingly in a future publication.

I express my gratitude to my advisor Prof.\ Wolfgang
Arendt for suggesting the problem, hinting towards the methods
I used, and in general for his valuable advises.

\section{Preliminaries}\label{sec:prelim}
In the whole article $\Omega$ will always denote a Lipschitz regular
subset of $\mathds{R}^N$, i.e., $\Omega$ is an open, bounded set that
is locally the epigraph of a Lipschitz regular function.
When we work with Lebesgue spaces $L^p(\partial\Omega)$, we always
equip $\partial\Omega$ with the natural surface measure, which coincides
the $(N-1)$-dimensional Hausdorff measure.
Since $\Omega$ is Lipschitz regular, there exists a bounded trace
operator from $H^1(\Omega)$ to $L^2(\partial\Omega)$, and
we denote the trace of $u \in H^1(\Omega)$ by $u|_{\partial\Omega}$
or simply by $u$, if misunderstandings are not to be expected.

We consider linear differential operators $L$ in divergence form
acting on functions on $\Omega$, i.e., $L$ is (formally) given by~\eqref{eq:diff_operator}.
We assume throughout that the coefficients $a_{ij}$, $b_j$, $c_i$ and $d$
are bounded and measurable and that $L$ is strictly elliptic, i.e.,
that there exists $\alpha > 0$ such that
\begin{equation}\label{eq:elliptic}
	\sum_{i,j=1}^N a_{ij}(x) \xi_i \xi_j \ge \alpha |\xi|^2
\end{equation}
holds for all $\xi \in \mathds{R}^N$ and almost every $x$ in $\Omega$.
Moreover, we restrict ourselves to the case $N \ge 2$ for simplicity.

For $L$ as in~\eqref{eq:diff_operator} and
$\beta \in L^\infty(\partial\Omega)$, define the bilinear form $a_{L,\beta}$ via
\begin{equation}\label{eq:robin_form}
	\begin{aligned}
		a_{L,\beta}(u,v) &
			\coloneqq \sum_{i,j=1}^N \int_\Omega a_{ij} D_iu D_jv \; \dx\lambda
				+ \sum_{j=1}^N \int_\Omega b_j u D_jv \; \dx\lambda \\
			& \qquad + \sum_{i=1}^N \int_\Omega c_i D_iu v \; \dx\lambda + \int_\Omega duv \; \dx\lambda
				+ \int_{\partial\Omega} \beta u v \; \dx\sigma
	\end{aligned}
\end{equation}
for $u$ and $v$ in $H^1(\Omega)$.

Given functions $f_j \in L^1(\Omega)$, $j=1,\dots,N$, and $g \in L^1(\partial\Omega)$,
we consider $u \in H^1(\Omega)$ that satisfy
\begin{equation}\label{eq:robin_problem}
	a_{L,\beta}(u,v) = \int_\Omega f_0 v \; \dx\lambda + \sum_{j=1}^N \int_\Omega f_j D_jv \; \dx\lambda + \int_{\partial\Omega} g v \; \dx\sigma
	\quad \text{ for all } v \in \mathrm{C}^1(\overline{\Omega}).
\end{equation}
This corresponds formally to the elliptic Robin problem~\eqref{eq:formal_problem}
with a distributional right-hand side.

If~\eqref{eq:robin_problem} holds maybe not for all $v \in \mathrm{C}^1(\overline{\Omega})$,
but at least for all smooth functions
$v \in \mathrm{C}^\infty_c(\Omega)$ with compact support in $\Omega$,
we say that $u \in H^1(\Omega)$ solves the problem
$Lu = f_0 - \sum_{j=1}^N D_jf_j$. Note that this condition does not depend on $\beta$.

In the proofs, we will frequently need Sobolev embedding theorems, which can be found for example
in Grisvard's book~\cite[Theorems~1.5.1.3 and~1.4.4.1]{Grisvard85}.

\section{Elliptic Problems}\label{sec:elliptic}
\subsection{Neumann boundary conditions}
In this section we consider~\eqref{eq:robin_problem} in the special
case $\beta = 0$, i.e., elliptic problems with Neumann boundary conditions.
We will see that for sufficiently regular right hand sides, every
solution admits a H\"older continuous representative.

Let $\Omega \subset \mathds{R}^N$ be Lipschitz regular.
By definition, for every $z \in \partial\Omega$ we can choose
an orthogonal matrix $\mathcal{O}$, a radius $r > 0$,
a Lipschitz continuous function $\psi\colon \mathds{R}^{N-1} \to \mathds{R}$ and
\[
	G \coloneqq \left\{ (y, \psi(y) + s) : y \in B(0,r) \subset \mathds{R}^{N-1}, s \in (-r,r) \right\}
\]
such that
\[
	\mathcal{O}(\Omega - z) \cap G = \left\{ (y, \psi(y) + s) : y \in B(0,r) \subset \mathds{R}^{N-1}, s \in (0,r) \right\}.
\]

\begin{convention}
	Since the assumptions of Section~\ref{sec:prelim} are invariant
	under isometric transformations of $\mathds{R}^N$,
	for local considerations
	we may without loss of generality assume that $\mathcal{O} = I$ and $z=0$.
\end{convention}

Define $T(y,s) \coloneqq (y, \psi(y) + s)$ for $y \in \mathds{R}^{N-1}$ and $s \in \mathds{R}$.
Then $T$ is a bi-Lipschitz mapping from $B(0,r) \times (-r,r)$ to $G$ with derivative
\[
	T'(y,s) = \begin{pmatrix} I & 0 \\ \nabla \psi(y) & 1 \end{pmatrix}
	\quad \text{and} \quad
	T'(y,s)^{-1} = \begin{pmatrix} I & 0 \\ -\nabla \psi(y) & 1 \end{pmatrix}
\]
almost everywhere.
Moreover, define the reflection $S\colon G \to G$ at the boundary
$\partial\Omega$ by $S(T(y,s)) \coloneqq T(y,-s)$.
Then
\[
	S'(T(y,s)) = T'(y,-s) \begin{pmatrix} I & 0 \\ 0 & -1 \end{pmatrix} T'(y,s)^{-1}
		= \begin{pmatrix} I & 0 \\ 2\nabla\psi(y) & -1 \end{pmatrix}
\]
almost everywhere.
Note that $S(Sx) = x$, $S'$ is bounded, $\det S'(x) = -1$ and $S'(x)^{-1} = S'(x)$.
Moreover, $S'(y,s)$ does not depend on $s$, whence $S'(Sx) = S'(x)$.

\begin{notation}
	We write $U$ for $G \cap \Omega$ and $V$ for $S(U) = G \setminus \overline{\Omega}$.
	For a function $w$ defined on $D \subset G$, define $w^\ast$ by 
	$w^\ast(x) \coloneqq w(Sx)$ on $S(D)$.
	If a function $w$ is defined on $U$, define $\tilde{w}$ on $G$ by
	\[
		\tilde{w}(x) \coloneqq \begin{cases} w(x), & x \in U, \\ w^\ast(x), & x \in V. \end{cases}
	\]
	on $G$.
	In the following it will not matter that $\tilde{w}$ is not defined on
	$\partial\Omega \cap G$ since we will apply this notation only to
	$L^p$-functions and the set has measure zero.

	For the rest of the section we fix a linear, strictly elliptic differential operator
	$L$ in divergence form and
	write $a$ for the matrix $(a_{ij})$ and $b$ and $c$ for the vectors
	$(b_j)$ and $(c_i)$, respectively.
	Moreover, define
	\begin{align*}
		\hat{a}(x) & \coloneqq
			\begin{cases}
				a(x), & x \in U, \\
				S'(x)a^\ast(x)S'(x)^T, & x \in V,
			\end{cases}
		&
		\hat{b}(x) & \coloneqq
			\begin{cases}
				b(x), & x \in U, \\
				S'(x)b^\ast(x), & x \in V,
			\end{cases}
		\\
		\hat{c}(x) & \coloneqq
			\begin{cases}
				c(x), & x \in U, \\
				S'(x)c^\ast(x), & x \in V.
			\end{cases}
	\end{align*}
\end{notation}

\begin{lemma}\label{lemma:prop}\mbox{}
	\begin{enumerate}[(i)]
	\item\label{lemma:prop:mirror}
		If $w$ is in $H^1(D)$, then $w^\ast$ is in $H^1(S(D))$, and $\nabla w^\ast(x) = \nabla w(Sx) S'(x)$ almost everywhere.
	\item\label{lemma:prop:trace}
		If $w$ is in $H^1(U)$, then $w|_{\partial U} = w^\ast|_{\partial V}$ on $\partial\Omega \cap G$.
	\item\label{lemma:prop:ext}
		If $w$ is in $H^1(U)$, then $\tilde{w}$ is in $H^1(G)$, and $\nabla \tilde{w} = \nabla w \, \setone_U + \nabla w^\ast \, \setone_V$.
	\item\label{lemma:prop:Lp}
		For any $p \in [1,\infty]$, the extension operator $w \mapsto \tilde{w}$ is continuous
		from $L^p(U)$ to $L^p(G)$.
	\item\label{lemma:prop:Linf}
		The functions $\hat{a}$, $\hat{b}$, $\hat{c}$ and $\tilde{d}$ are measurable and bounded on $G$.
	\end{enumerate}
\end{lemma}
\begin{proof}
	Assertion~\eqref{lemma:prop:mirror} follows from \cite[Theorem~2.2.2]{Ziemer89}.
	Assertion~\eqref{lemma:prop:trace} is obvious if $w$ is in addition continuous up to
	the boundary. Since $U$ is Lipschitz regular, those functions are dense in $H^1(U)$
	and the claim follows by approximation.
	Let $\phi$ be a test function on $G$. The divergence theorem~\cite[\S 4.3]{EG92}
	shows that
	\[
		\int_G \tilde{w} \; D_i\phi \; \dx\lambda
			= \int_{\partial U} w \phi \; \nu_i \; \dx\sigma - \int_U D_iw \; \phi \; \dx\lambda
				+ \int_{\partial V} w^\ast \phi \; \nu_i \; \dx\sigma - \int_V D_iw^\ast \; \phi \; \dx\lambda
	\]
	The boundary integrals cancel due to~\eqref{lemma:prop:trace} since the boundaries
	$\partial V$ and $\partial U$ have opposite orientations.
	This proves~\eqref{lemma:prop:ext}.
	Assertion~\eqref{lemma:prop:Lp} follows from~\cite[\S 3.4.3]{EG92}, and
	assertion~\eqref{lemma:prop:Linf} is obvious.
\end{proof}

\begin{lemma}\label{lemma:ext_elliptic}
	There exists a constant $\hat{\alpha} > 0$ such that $\xi^T \hat{a}(x)\xi \ge \hat{\alpha} |\xi|^2$
	for all $\xi \in \mathds{R}^N$ and almost every $x \in G$.
\end{lemma}
\begin{proof}
	Let $w \in \mathds{R}^{N-1}$ be an arbitrary row vector and define
	$W \coloneqq \begin{smallpmatrix} I & 0 \\ w & -1 \end{smallpmatrix}$.
	Given a positive definite matrix
	$M \coloneqq \begin{smallpmatrix} A & b \\ c & d \end{smallpmatrix} \in \mathds{R}^{N \times N}$,
	the matrix
	\[
		WMW^T = \begin{pmatrix} A & Aw^T-b \\ wA-c & wAw^T - wb - cw^T + d \end{pmatrix}
	\]
	is positive definite as well. In fact, it suffices to check that the
	leading principal minors are positive. Since $M$ is positive definite,
	all minors of $M$ are positive. Hence the first $N-1$ leading
	principal minors of $WMW^T$ are positive and, moreover, $\det M > 0$.
	Thus $\det(WMW^T) > 0$ by the multiplicativity of the determinant
	since $\det W = \det W^T = -1$, which proves the claim.
	
	By what we have shown, the least eigenvalue $\lambda_1(WMW^T)$ of $WMW^T$
	is positive whenever $M$ is positive definite.
	Since $\lambda_1$ depends continuously
	on the entries of the matrix this shows that $\lambda_1(WMW^T) \ge \delta$
	for some $\delta > 0$ as $M$ ranges over a compact subset of the set of
	all positive definite matrices, and $w$ ranges over a compact subset of $\mathds{R}^{N-1}$.

	Recall that a matrix $A \in \mathds{R}^{N \times N}$ satisfies
	$\xi^T A\xi \ge \alpha |\xi|^2$, $\alpha > 0$, for all $\xi \in \mathds{R}^N$
	if and only if $\lambda_1((A+A^T)/2) \ge \alpha$.
	Thus by assumption~\eqref{eq:elliptic}
	\[
		\tfrac{1}{2} \bigl(a(x) + a(x)^T\bigr) \in K_1 \coloneqq \left\{ M \in \mathds{R}^{N \times N} : M = M^T, \lambda_1(M) \ge \alpha, \|M\| \le c \right\}
	\]
	for some constant $c$ and for almost all $x \in U$.
	The set $K_1$ is a compact subset of the positive definite matrices.
	Let $K_2 \subset \mathds{R}^{N-1}$ be a closed ball whose radius is larger enough
	such that $2\nabla\psi(y) \in K_2$ for almost all $y$.

	Using the first part of this proof, we see that there is $\delta > 0$ such that
	\[
		\lambda_1\bigl(\tfrac{1}{2} S'(x) \bigl( a(Sx) + a(Sx)^T \bigr) S'(x)^T \bigr) \ge \delta
	\]
	for almost every $x \in U$. Thus $\xi^T \hat{a}(x) \xi \ge \delta |\xi|^2$
	almost everywhere on $V$, from which the claim follows with
	$\hat{\alpha} \coloneqq \min\{\alpha, \delta\}$.
\end{proof}

\begin{lemma}\label{lemma:ext_solution}
	Let $\hat{L}$ denote the differential operator on $G$
	for the coefficients $\hat{a}$, $\hat{b}$, $\hat{c}$ and $\tilde{d}$.
	Assume that there exists $p > N$ such that
	$f_0 \in L^{p/2}(\Omega)$, $f_j \in L^p(\Omega)$, $j=1,\dots,N$, and
	$g \in L^{p-1}(\partial\Omega)$.
	Assume that $u \in H^1(\Omega)$ is a solution of~\eqref{eq:robin_problem} (recall that we
	allow only for $\beta = 0$ in this section).
	Then there exist $s>N$ and functions $h_0 \in L^{s/2}$ and $h_j \in L^s$, $j=1,\dots,N$,
	that satisfy
	\begin{equation}\label{eq:interior_equation:final}
		a_{\hat{L},0}(\tilde{u},v) = \int_G h_0 v \, \dx\lambda + \sum_{i=1}^N \int_G h_i \; D_iv \; \dx\lambda
	\end{equation}
	for every function $v \in \mathrm{C}^\infty_c(G)$.
\end{lemma}
\begin{proof}\allowdisplaybreaks
	By definition of a solution of~\eqref{eq:robin_problem} we have that
	\begin{align*}
		& \sum_{i,j=1}^N \int_U \hat{a}_{ij} D_i\tilde{u} D_j v \; \dx\lambda
			+ \sum_{j=1}^N \int_U \hat{b}_j \tilde{u} D_jv \; \dx\lambda
			+ \sum_{i=1}^N \int_U \hat{c}_i D_i\tilde{u} v \; \dx\lambda
			+ \int_U \tilde{d} \tilde{u} v \; \dx\lambda \\
		& \quad = \int_U f_0 v \; \dx\lambda + \sum_{j=1}^N \int_U f_j D_jv \; \dx\lambda + \int_{\partial U} g v \; \dx\sigma
	\end{align*}
	holds for every $v \in \mathrm{C}^\infty_c(G)$.

	Using part~\eqref{lemma:prop:mirror} of Lemma~\ref{lemma:prop} and the change of
	variables formula~\cite[\S 3.4.3]{EG92}, replacing $x$ by $Sx$, we obtain
	\begin{align*}
		& \int_V (\nabla \tilde{u}) \hat{a} (\nabla v)^T \; \dx\lambda
			+ \int_V \tilde{u} (\nabla v) \hat{b} \; \dx\lambda
			+ \int_V (\nabla \tilde{u}) \hat{c} v \; \dx\lambda
			+ \int_V \tilde{d} \tilde{u} v \; \dx\lambda \\
		& = \int_V \nabla u(Sx) S'(x) S'(x) a(Sx) S'(x)^T (\nabla v(x))^T \dx x \\
			& \qquad + \int_V u(Sx) \nabla v(x) S'(x) b(Sx) \dx x
			+ \int_V \nabla u(Sx) S'(x) S'(x) c(Sx) v(x) \dx x \\
			& \qquad + \int_V d(Sx) u(Sx) v(x) \dx x \\
		& = \int_U \nabla u(x) a(x) S'(x)^T (\nabla v(Sx))^T \dx x
			+ \int_U u(x) \nabla v(Sx) S'(x) b(x) \dx x \\
			& \qquad + \int_U \nabla u(x) c(x) v(Sx) \dx x
			+ \int_U d(x) u(x) v(Sx) \dx x \\
		& = \int_U \nabla u(x) a(x) (\nabla v^\ast(x))^T \dx x
			+ \int_U u(x) \nabla v^\ast(x) b(x) \dx x \\
			& \qquad + \int_U \nabla u(x) c(x) v^\ast(x) \dx x
			+ \int_U d(x) u(x) v^\ast(x) \dx x \\
		& = \int_U f_0 v^\ast \; \dx\lambda + \sum_{j=1}^N \int_U f_j D_jv^\ast \; \dx\lambda + \int_{\partial U} g v^\ast \; \dx\sigma \\
		& = \int_U f_0(x) v(Sx) \dx x + \sum_{j=1}^N \int_U f_j(x) \sum_{i=1}^N D_iv(Sx) (S'(x))_{ij} \dx x + \int_{\partial U} g v^\ast \; \dx\sigma \\
		& = \int_V f_0^\ast v \; \dx\lambda + \sum_{i=1}^N \int_V \sum_{j=1}^N (S')_{ij} f_j^\ast \; D_i v \; \dx\lambda + \int_{\partial U} g v \; \dx\sigma,
	\end{align*}
	for every $v \in \mathrm{C}^\infty_c(\Omega)$,
	where we have used that $u$ is solution of~\eqref{eq:robin_problem}.

	Adding these two equations and defining
	$\hat{f}_j \in L^p(G)$ by
	\[
		\hat{f}_j(x) \coloneqq
			\begin{cases}
				f_j(x), & x \in U, \\
				\sum_{i=1}^N (S'(x))_{ji} f_i^\ast(x), & x \in V,
			\end{cases}
	\]
	we obtain
	\begin{equation}\label{eq:interior_equation:boundary}
		a_{\hat{L},0}(\tilde{u}, v)
			= \int_G \tilde{f}_0 v \; \dx\lambda + \sum_{j=1}^N \int_G \hat{f}_j D_jv \; \dx\lambda + 2 \int_{\partial U} gv \; \dx\sigma.
	\end{equation}

	Since $g$ is in $L^{p-1}(\partial U)$ and the trace operator
	is bounded from $W^{1,r}(G)$ to
	$L^{(N-1)r/(N-r)}(\partial U)$ for every $r \in (1,N)$
	the mapping
	\[
		\mathrm{C}^\infty_c(G) \to \mathds{R}, \; v \mapsto 2 \int_{\partial U} g v \; \dx\sigma
	\]
	extends to a continuous linear functional on $W^{1,r_0}_0(G)$
	for $r_0 \coloneqq \frac{(p-1)N}{(p-2)N + 1}$.
	Thus there exist functions $(k_j)_{j=0}^N$ in $L^{r_0'}(G)$
	such that
	\[
		2 \int_{\partial U} g v \; \dx\sigma = \int_G k_0 v \; \dx\lambda + \sum_{j=1}^N \int_G k_j D_jv \; \dx\lambda
	\]
	holds for every test function $v$, cf.~\cite[Theorem~4.3.3]{Ziemer89}.
	Note that by assumption $r_0' = \frac{(p-1)N}{N-1}$ is larger than $N$.
	Hence~\eqref{eq:interior_equation:final} follows from~\eqref{eq:interior_equation:boundary}
	by setting $h_0 \coloneqq \tilde{f}_0 + k_0$ and $h_j \coloneqq \hat{f}_j + k_j$ for $j=1,\dots,N$.
\end{proof}

\begin{proposition}\label{prop:neumann_Creg}
	Let $\Omega \subset \mathds{R}^N$ be Lipschitz regular, $p > N$.
	There exist $\gamma > 0$ and a constant $c$ with the following property.
	If $f_0 \in L^{p/2}(\Omega)$, $f_j \in L^p(\Omega)$, $j=1,\dots,N$, and $g \in L^{p-1}(\partial\Omega)$,
	then every solution $u$ of~\eqref{eq:robin_problem} (recall that at the moment we allow only for $\beta=0$)
	is in $\mathrm{C}^{0,\gamma}(\Omega)$ and satisfies
	\begin{equation}\label{eq:neumann_Creg:est}
		\|u\|_{\mathrm{C}^{0,\gamma}(\Omega)}
			\le c \bigl(\|u\|_{L^2(\Omega)} + \|f_0\|_{L^{p/2}(\Omega)} + \sum_{j=1}^N \|f_j\|_{L^p(\Omega)} + \|g\|_{L^{p-1}(\partial\Omega)}\bigr).
	\end{equation}
\end{proposition}

\begin{proof}
	Fix $z$ and $G$ as at the beginning of this section.
	By Lemma~\ref{lemma:ext_solution} there exists $s > N$ and functions
	$h_0 \in L^{s/2}(G)$ and $h_j \in L^s(\Omega)$, $j=1,\dots,N$,
	such that the extension $\tilde{u} \in H^1(G)$ of $u$ solves the problem
	\[
		\hat{L}\tilde{u} = h_0 - \sum_{j=1}^N D_jh_j.
	\]
	By Lemmata~\ref{lemma:prop} and~\ref{lemma:ext_elliptic},
	the differential operator $\hat{L}$ on $G$ satisfies the
	assumptions of Section~\ref{sec:prelim}.
	Thus it follows from results about interior regularity~\cite[Theorem~8.24]{GT01}
	for every $\omega \Subset G$ there exists $\gamma_0 > 0$ such that
	$\tilde{u}$ is in $\mathrm{C}^{0,\gamma_0}(\omega)$ and satisfies an estimate
	of the same kind as~\eqref{eq:neumann_Creg:est}.
	Thus $u$ is in $\mathrm{C}^{0,\gamma}(\omega \cap \Omega)$ and satisfies
	an appropriate estimate.

	Since $\partial\Omega$ is compact, we can cover $\partial\Omega$ by
	finitely many such sets.
	Using the result about interior regularity once again to control
	$u$ in the remaining part of $\Omega$, the result follows.
\end{proof}

\subsection{Robin Boundary Conditions}
In this section we will apply Proposition~\ref{prop:neumann_Creg} to obtain
similar results also for Robin boundary conditions, i.e., for solutions
of~\eqref{eq:robin_problem} if $\beta$ does not necessarily equal $0$.
As a stepping stone, we investigate the $L^p$-regularity of these solutions
also in cases where the data satisfies fewer regularity
assumptions than in Proposition~\ref{prop:neumann_Creg}.
Thus even for $\beta=0$, the results of this section are more general
than those of the previous one.

As before, let $\Omega \subset \mathds{R}^N$ be Lipschitz regular
and $L$ be a linear, strictly elliptic differential operator on $\Omega$.
Moreover, let $\beta$ be an arbitrary function in $L^\infty(\partial\Omega)$.

For $\omega \in \mathds{R}$ we introduce the forms $a_{L,\beta}^\omega$, which are defined via
\begin{equation}\label{eq:robin_form:omega}
	a_{L,\beta}^\omega(u,v) \coloneqq a_{L,\beta}(u,v) + \omega \int_\Omega u v \; \dx\lambda
\end{equation}
for $u$ and $v$ in $H^1(\Omega)$.
We consider the functions $u \in H^1(\Omega)$ that satisfy
\begin{equation}\label{eq:robin_problem:omega}
	a_{L,\beta}^\omega(u,v) = \int_\Omega f_0 v \; \dx\lambda + \sum_{j=1}^N \int_\Omega f_j D_jv \; \dx\lambda + \int_{\partial\Omega} g v \; \dx\sigma
	\quad \text{ for all } v \in \mathrm{C}^1(\overline{\Omega}).
\end{equation}
This is a generalized version of~\eqref{eq:robin_problem},
and these two problems coincide for $\omega=0$.
The advantage of the more general form is that for large
$\omega$ the problem~\eqref{eq:robin_problem:omega} is uniquely solvable.

\begin{lemma}\label{lem:robin_laxmilgram}
	Let $N \ge 3$.
	There exist $\omega \in \mathds{R}$ and a constant $c$ with the following property.
	If $f_0 \in L^{2N/(N+2)}(\Omega)$, $f_j \in L^2(\Omega)$, $j=1,\dots,N$, and $g \in L^{2(N-1)/N}(\partial\Omega)$,
	then problem~\eqref{eq:robin_problem:omega} has a unique solution $u \in H^1(\Omega)$, and
	\begin{equation}\label{eq:robin_laxmilgram:est}
		\|u\|_{H^1(\Omega)}
			\le c \bigl( \|f_0\|_{L^{2N/(N+2)}(\Omega)} + \sum_{j=1}^N \|f_j\|_{L^2(\Omega)} + \|g\|_{L^{2(N-1)/N}(\partial\Omega)}\bigr).
	\end{equation}
\end{lemma}
\begin{proof}
	By~\cite[Corollary~2.5]{Daners09}
	there exist $\eta > 0$ and $\omega \in \mathds{R}$ such that
	\[
		a_{L,\beta}^\omega(u,u) \ge \eta \|u\|_{H^1(\Omega)}^2.
	\]
	Thus, by the Lax-Milgram theorem~\cite[Theorem~5.8]{GT01} there exists
	a constant $c_1$ with the following property.
	For every $\psi \in H^1(\Omega)'$
	there exists a unique function $u \in H^1(\Omega)$ that satisfies
	\begin{equation}\label{eq:robin_laxmilgram:repr}
		a_{L,\beta}^\omega(u,v) = \psi(v)
		\quad\text{for all } v \in H^1(\Omega),
	\end{equation}
	and for this $u$ we have
	\begin{equation}\label{eq:robin_laxmilgram:H1est}
		\| u \|_{H^1(\Omega)} \le c_1 \| \psi \|_{H^1(\Omega)'}.
	\end{equation}

	Since $H^1(\Omega)$ embeds into $L^{2N/(N-2)}(\Omega)$
	and the trace operator maps $H^1(\Omega)$
	into $L^{2(N-1)/(N-2)}(\partial\Omega)$,
	there exists a constant $c_2$ with the following property.
	For $f_0 \in L^{2N/(N+2)}(\Omega)$, $f_j \in L^2(\Omega)$, $j=1,\dots,N$,
	and $g \in L^{2(N-1)/N}(\partial\Omega)$,
	\begin{equation}\label{eq:robin_laxmilgram:func}
		\psi(v) \coloneqq 
			\int_\Omega f_0 v \; \dx\lambda + \sum_{j=1}^N \int_\Omega f_j D_jv \; \dx\lambda + \int_{\partial\Omega} g v \; \dx\sigma
	\end{equation}
	defines a continuous linear functional $\psi$ on $H^1(\Omega)$ that satisfies
	\begin{equation}\label{eq:robin_laxmilgram:funcest}
		\|\psi\|_{H^1(\Omega)'}
			\le c_2 \bigl(\|f_0\|_{L^{2N/(N+2)}(\Omega)} + \sum_{j=1}^N \|f_j\|_{L^2(\Omega)} + \|g\|_{L^{2(N-1)/N}(\partial\Omega)}\bigr).
	\end{equation}

	Now let $f_0 \in L^{2N/(N+2)}(\Omega)$, $f_j \in L^2(\Omega)$, $j=1,\dots,N$,
	and $g \in L^{2(N-1)/N}(\partial\Omega)$ be arbitrary.
	Define $\psi$ as in~\eqref{eq:robin_laxmilgram:func},
	and let $u$ be as in~\eqref{eq:robin_laxmilgram:repr}.
	Then $u$ is a solution of~\eqref{eq:robin_problem:omega}.
	Since $\mathrm{C}^1(\overline{\Omega})$ is dense in $H^1(\Omega)$,
	every solution of~\eqref{eq:robin_problem:omega} satisfies~\eqref{eq:robin_laxmilgram:repr}.
	Hence the solution of~\eqref{eq:robin_problem:omega} is unique.
	Estimate~\eqref{eq:robin_laxmilgram:est} follows with $c \coloneqq c_1c_2$
	by combining~\eqref{eq:robin_laxmilgram:H1est}
	and~\eqref{eq:robin_laxmilgram:funcest}.
\end{proof}

\begin{lemma}\label{lem:robin_laxmilgram2}
	Let $N=2$ and $q > 1$.
	There exist $\omega \in \mathds{R}$ and a constant $c$ with the following property.
	If $f_0 \in L^q(\Omega)$, $f_j \in L^2(\Omega)$, $j=1,\dots,N$, and $g \in L^q(\partial\Omega)$,
	then problem~\eqref{eq:robin_problem:omega} has a unique solution $u \in H^1(\Omega)$, and
	\begin{equation}\label{eq:robin_laxmilgram2:est}
		\|u\|_{H^1(\Omega)}
			\le c \bigl( \|f_0\|_{L^q(\Omega)} + \sum_{j=1}^N \|f_j\|_{L^2(\Omega)} + \|g\|_{L^q(\partial\Omega)}\bigr).
	\end{equation}
\end{lemma}
\begin{proof}
	The proof is similar to the proof of Lemma~\ref{lem:robin_laxmilgram}.
	Here, however, we use that $H^1(\Omega)$ embeds into $L^r(\Omega)$ for
	every $r < \infty$, and that the trace operator maps $H^1(\Omega)$ into
	$L^r(\partial\Omega)$ for every $r < \infty$.
\end{proof}

\begin{remark}\label{rem:fredholm}
	It should be noted that in Lemmata~\ref{lem:robin_laxmilgram}
	and~\ref{lem:robin_laxmilgram2} we can take any $\omega \in \mathds{R}$
	such that~\eqref{eq:robin_problem:omega} has a unique solution for
	some right hand side.
	In fact, let $A$ be the operator from $H^1(\Omega)$ to $H^1(\Omega)'$
	defined by $\apply{Au}{v} \coloneqq a_{L,\beta}(u,v)$.
	Considering $H^1(\Omega)$ as a subspace of $H^1(\Omega)'$ via the scalar
	product in $L^2(\Omega)$, $A$ is a densely defined, closed operator
	on $H^1(\Omega)'$. The resolvent of $A$ is compact since $H^1(\Omega)$
	is compactly embedded into $L^2(\Omega)$.
	For $\omega \in \mathds{R}$, the Fredholm alternative asserts that either
	there exists $u \in H^1(\Omega)$ such that $(\omega + A)u = 0$, which means precisely that
	the solution of~\eqref{eq:robin_problem:omega} is not unique,
	or $\omega + A$ is boundedly invertible, which implies
	estimate~\eqref{eq:robin_laxmilgram:est} or~\eqref{eq:robin_laxmilgram2:est}, respectively.
\end{remark}

Now, as an interlude, we come back to the Neumann problem.
Afterwards, the following propositions will be generalized to cover
Robin problems as well.
\begin{lemma}\label{lem:neumann_Creg_inv}
	Let $p > N$, and let $\omega$ be as in Lemma~\ref{lem:robin_laxmilgram} or
	Lemma~\ref{lem:robin_laxmilgram2}, respectively.
	Assume $\beta = 0$. Then there exist $\gamma > 0$ and a constant $c$ with the following property.
	If $f_0 \in L^{p/2}(\Omega)$, $f_j \in L^p(\Omega)$, $j=1,\dots,N$, and $g \in L^{p-1}(\partial\Omega)$,
	then the unique solution $u$ of~\eqref{eq:robin_problem:omega} is in $\mathrm{C}^{0,\gamma}(\Omega)$ and satisfies
	\[
		\|u\|_{\mathrm{C}^{0,\gamma}(\Omega)}
			\le c \bigl(\|f_0\|_{L^{p/2}(\Omega)} + \sum_{j=1}^N \|f_j\|_{L^p(\Omega)} + \|g\|_{L^{p-1}(\partial\Omega)}\bigr).
	\]
\end{lemma}
\begin{proof}
	By Lemma~\ref{lem:robin_laxmilgram} or Lemma~\ref{lem:robin_laxmilgram2},
	respectively, the solution is unique, and by~\eqref{eq:robin_laxmilgram:est}
	or~\eqref{eq:robin_laxmilgram2:est} there exists
	a constant $c_1$ such that
	\[
		\| u \|_{L^2(\Omega)} \le c_1 \bigl(\|f_0\|_{L^{p/2}(\Omega)} + \sum_{j=1}^N \|f_j\|_{L^p(\Omega)} + \|g\|_{L^{p-1}(\partial\Omega)}\bigr).
	\]
	Thus the result follows from Proposition~\ref{prop:neumann_Creg}.
\end{proof}

\begin{lemma}\label{lem:neumann_lpreg}
	Let $N \ge 3$, $\frac{2N}{N+2} \le q < \frac{N}{2}$, $\eps > 0$, and let
	$\omega$ be as in Lemma~\ref{lem:robin_laxmilgram}.
	Then there exists a constant $c$ with the following property.
	If $f_0 \in L^{q+\eps}(\Omega)$, $f_j \in L^{Nq/(N-q) + \eps}(\Omega)$, $j=1,\dots,N$, and
	$g \in L^{(N-1)q/(N-q) + \eps}(\partial\Omega)$, then the unique solution
	$u$ of~\eqref{eq:robin_problem:omega} satisfies 
	$u \in L^{Nq/(N-2q)}(\Omega)$ and
	$u|_{\partial\Omega} \in L^{(N-1)q/(N-2q)}(\partial\Omega)$, and
	\begin{align*}
		& \|u\|_{L^{Nq/(N-2q)}(\Omega)} + \|u\|_{L^{(N-1)q/(N-2q)}(\partial\Omega)} \\
			& \quad \le c \bigl(\|f_0\|_{L^{q+\eps}(\Omega)}
				+ \sum_{j=1}^N \|f_j\|_{L^{Nq/(N-q)+\eps}(\Omega)} + \|g\|_{L^{(N-1)q/(N-q)+\eps}(\partial\Omega)}\bigr)
	\end{align*}
\end{lemma}
\begin{proof}
	Pick $p > N$. It will turn out at the end how close to $N$ we have to pick $p$,
	but this condition will depend only on $N$ and $q$, hence the argument is not circular.

	To simplify the notation of the proof we introduce the Banach spaces
	\[
		L^{r,s,t} \coloneqq L^r(\Omega) \oplus L^s(\Omega)^N \oplus L^t(\partial\Omega)
		\quad\text{and}\quad
		L^{x,y} \coloneqq L^x(\Omega) \oplus L^y(\partial\Omega).
	\]
	Note that the complex interpolation spaces
	$[ L^{r_0,s_0,t_0}, L^{r_1,s_1,t_1} ]_\theta$ and
	$[ L^{x_0,y_0}, L^{x_1,y_1} ]_\theta$, $\theta \in [0,1]$, are by the natural isomorphism
	isomorphic to $L^{r,s,t}$ and $L^{x,y}$, respectively, where
	\begin{align}
		\tfrac{1}{r} & = \tfrac{1-\theta}{r_0} + \tfrac{\theta}{r_1}, &
		\tfrac{1}{s} & = \tfrac{1-\theta}{s_0} + \tfrac{\theta}{s_1}, &
		\tfrac{1}{t} & = \tfrac{1-\theta}{t_0} + \tfrac{\theta}{t_1}, \label{eq:neumann:lpreg:conj} \\
		\tfrac{1}{x} & = \tfrac{1-\theta}{x_0} + \tfrac{\theta}{x_1}, &
		\tfrac{1}{y} & = \tfrac{1-\theta}{y_0} + \tfrac{\theta}{y_1}. \nonumber
	\end{align}
	This follows from~\cite[\S 1.18.4]{Triebel95} and the observation that
	\[
		\bigl[ X_0 \oplus Y_0, X_1 \oplus Y_1 \bigr]_\theta
			\cong \bigl[ X_0, X_1 \bigr]_\theta \oplus \bigl[ Y_0, Y_1 \bigr]_\theta
	\]
	holds for all Banach spaces $X_0$, $X_1$, $Y_0$ and $Y_1$, which is
	a direct consequence of the definition of the complex interpolation
	functor~\cite[\S 1.9]{Triebel95}.

	For $f_0 \in L^{2N/(N+2)}(\Omega)$, $f_j \in L^2(\Omega)$,
	$j=1,\dots,N$, and $g \in L^{2(N-1)/N}(\partial\Omega)$ we denote
	by $R(f_0, (f_j)_{j=1}^N, g)$ the unique solution $u \in H^1(\Omega)$
	of~\eqref{eq:robin_problem:omega}.  It is clear that $R$ is a linear map.
	Let $\gamma > 0$ be as in Lemma~\ref{lem:neumann_Creg_inv}.
	If we consider $H^1(\Omega)$ and $\mathrm{C}^{0,\gamma}(\Omega)$
	as subspaces of $L^{2,2}$ via the injection $u \mapsto (u, u|_{\partial\Omega})$,
	then the Sobolev embedding theorems
	and Lemmata~\ref{lem:robin_laxmilgram} and~\ref{lem:neumann_Creg_inv} show that
	$R$ maps $L^{2N/(N+2),2,2(N-1)/N}$ into
	$L^{2N/(N-2),2(N-1)/(N-2)}$ and in addition $L^{p/2,p,p-1}$ into $L^{\infty,\infty}$.

	Using~\cite[Theorem~1.9.3(a)]{Triebel95} for
	\[
		\theta \coloneqq \tfrac{Nq + 2q - 2N}{q(N-2)},
	\]
	we obtain that $R$ maps
	$L^{r_p,s_p,t_p}$ into $L^{Nq/(N-2q),(N-1)q/(N-2q)}$,
	where $r_p$, $s_p$ and $t_p$ are defined as in~\eqref{eq:neumann:lpreg:conj} for
	\begin{align*}
		r_0 & = \tfrac{2N}{N+2}, &
		s_0 & = 2, &
		t_0 & = \tfrac{2(N-1)}{N}, \\
		r_1 & = \tfrac{p}{2}, &
		s_1 & = p, &
		t_1 & = p-1.
	\end{align*}
	It is easy to see that the dependence of $r_p$, $s_p$ and $t_p$ on $p$
	is continuous and that
	\[
		r_N = q, \quad
		s_N = \tfrac{Nq}{N-q} \quad
		\text{and} \quad
		t_N = \tfrac{(N-1)q}{N-q}.
	\]
	Thus there exists $p > N$ such that
	\[
		r_p < q + \eps, \quad
		s_p < \tfrac{Nq}{N-q} + \eps \quad
		\text{and} \quad
		t_p < \tfrac{(N-1)q}{N-q} + \eps.
	\]
	The result follows if we start the whole argument with such a value for $p$.
\end{proof}

\begin{remark}
	We exclude $N=2$ in Lemma~\ref{lem:neumann_lpreg}
	because the admissible range for $q$ is empty in that case.
	However, if we take $N=2$ and $q=1$ and adopt the convention that
	$\frac{1}{0}$ be $\infty$, Lemma~\ref{lem:neumann_lpreg} is a special case
	of Lemma~\ref{lem:neumann_Creg_inv}.
	More generally, this is true also for $N \ge 3$ in the boundary
	case $q = \frac{N}{2}$.
\end{remark}

Now we come back to Robin boundary conditions.
The following bootstrapping argument allows us to deduce regularity
results for Robin problems from the corresponding results
for Neumann problems.

\begin{lemma}\label{lem:robin_Lpreg}
	Let $N \ge 3$, $\frac{2N}{N+2} \le q \le \frac{N}{2}$ and $\eps > 0$.
	Then there exist $\tilde{\eps} > 0$ and a constant $c$ with the following property.
	If $f_0 \in L^{q+\eps}(\Omega)$, $f_j \in L^{Nq/(N-q) + \eps}(\Omega)$, $j=1,\dots,N$, and
	$g \in L^{(N-1)q/(N-q) + \eps}(\partial\Omega)$, then every solution
	$u \in H^1(\Omega)$ of~\eqref{eq:robin_problem} satisfies 
	$u \in L^{q+\tilde{\eps}}(\Omega)$, $u|_{\partial\Omega} \in L^{(N-1)q/(N-q) + \tilde{\eps}}(\partial\Omega)$ and
	\begin{align*}
		\|u\|_{L^{q+\tilde{\eps}}(\Omega)}
			& \le c \bigl(\|u\|_{L^2(\Omega)} + \|f_0\|_{L^{q+\eps}(\Omega)} + \sum_{j=1}^N \|f_j\|_{L^{Nq/(N-q)+\eps}(\Omega)} \\
			& \qquad + \|g\|_{L^{(N-1)q/(N-q)+\eps}(\partial\Omega)}\bigr).
	\end{align*}
\end{lemma}
\begin{proof}
	Define by induction
	\[
		q_0 \coloneqq \tfrac{2N}{N+2},
		\quad\text{and}\quad
		q_{n+1} \coloneqq \min\bigl\{ q, \tfrac{Nq_n}{N-2q_n} \bigr\},
	\]
	where we adopt the convention that $1/0 \coloneqq \infty$.
	Note that there exists $n_0 \in \mathds{N}_0$ such that $q_n = q$, since otherwise
	we would have
	\[
		q_n = \tfrac{Nq_{n-1}}{N-2q_{n-1}} \ge \tfrac{N}{N-2} q_{n-1} \ge \dots \ge \bigl( \tfrac{N}{N-2} \bigr)^n q_0 \to \infty \quad (n \to \infty)
	\]
	which is not possible since $q_n \le q$ for all $n \in \mathds{N}$ by definition.

	For $n \in \mathds{N}_0$, we say that $(P_n)$ is fulfilled
	if there exist $\eps_n > 0$ and a constant $c_n$ with the following property.
	If $f_0 \in L^{q+\eps}(\Omega)$, $f_j \in L^{Nq/(N-q) + \eps}(\Omega)$, $j=1,\dots,N$, and
	$g \in L^{(N-1)q/(N-q) + \eps}(\partial\Omega)$, then every solution
	$u \in H^1(\Omega)$ of~\eqref{eq:robin_problem} satisfies 
	$u \in L^{q_n+\eps_n}(\Omega)$, $u|_{\partial\Omega} \in L^{(N-1)q_n/(N-q_n) + \eps_n}(\partial\Omega)$ and
	\begin{equation}\label{eq:robin_Lpreg:indest}
		\begin{aligned}
		\|u\|_{L^{q_n+\eps_n}(\Omega)}
			& \le c_n \bigl(\|u\|_{L^2(\Omega)} + \|f_0\|_{L^{q+\eps}(\Omega)}
			+ \sum_{j=1}^N \|f_j\|_{L^{Nq/(N-q)+\eps}(\Omega)} \\
			& \qquad + \|g\|_{L^{(N-1)q/(N-q)+\eps}(\partial\Omega)}\bigr).
		\end{aligned}
	\end{equation}

	The statement $(P_0)$ is obviously true.
	So now assume that $(P_n)$ is true for some $n \in \mathds{N}_0$.
	If $q_n = q$, then $(P_{n+1})$ is trivially fulfilled since
	it is the same statement as $(P_n)$.
	Thus we may assume $q_n < q$ without loss of generality.
	Let $\omega$ be as in Lemma~\ref{lem:robin_laxmilgram} or
	Lemma~\ref{lem:robin_laxmilgram2}, respectively,
	and note that every solution $u \in H^1(\Omega)$ of~\eqref{eq:robin_problem}
	satisfies
	\begin{equation}\label{eq:robin_Lpreg:modeq}
		a_{L,\beta}^\omega(u,v) = \int_\Omega (\omega u + f_0) v \; \dx\lambda + \sum_{j=1}^N \int_\Omega f_j D_jv \; \dx\lambda + \int_{\partial\Omega} g v \; \dx\sigma
	\end{equation}
	for all $v \in \mathrm{C}^1(\overline{\Omega})$, i.e., $u$ solves~\eqref{eq:robin_problem:omega}
	for a different right-hand side that involves $u$.
	Thus Lemma~\ref{lem:neumann_lpreg} applied to a value $\tilde{q}$ such that
	$q_n < \tilde{q} < \min\{ q, q_n + \eps_n \}$ implies
	that there exist constants $\tilde{c}$ and $\eps_{n+1} > 0$ such that
	\begin{align*}
		\|u\|_{L^{q_{n+1}+\eps_{n+1}}(\Omega)}
			& \le \tilde{c} \bigl(\|u\|_{L^{q_n+\eps_n}(\Omega)} + \|f_0\|_{L^{q+\eps}(\Omega)}
			+ \sum_{j=1}^N \|f_j\|_{L^{Nq/(N-q)+\eps}(\Omega)} \\
			& \qquad + \|g\|_{L^{(N-1)q/(N-q)+\eps}(\partial\Omega)}\bigr).
	\end{align*}
	Using $(P_n)$ to estimate $\|u\|_{L^{q_n+\eps_n}(\Omega)}$ as in~\eqref{eq:robin_Lpreg:indest},
	we have proved $(P_{n+1})$.
	By induction, $(P_n)$ is true for every $n \in \mathds{N}_0$.
	Since the statement of the lemma is equivalent to $(P_{n_0})$ for some $n_0 \in \mathds{N}_0$
	such that $q_{n_0} = q$, this finishes the proof.
\end{proof}

The following theorem summarizes (and extends) all previous results in this section
with the exception of Lemmata~\ref{lem:robin_laxmilgram} and~\ref{lem:robin_laxmilgram2},
which are slightly sharper.
As always, we assume that $\Omega \subset \mathds{R}^N$, $N \ge 2$, is Lipschitz regular
and that $L$ is a strictly elliptic differential operator with bounded coefficients.
\begin{theorem}\label{thm:robin_reg}
	Let $\frac{2N}{N+2} \le q \le \frac{N}{2}$, $\eps > 0$.
	Then there exist $\gamma > 0$ and a constant $c$ with the following property.
	If $f_0 \in L^{q+\eps}(\Omega)$, $f_j \in L^{Nq/(N-q) + \eps}(\Omega)$, $j=1,\dots,N$, and
	$g \in L^{(N-1)q/(N-q) + \eps}(\partial\Omega)$, then every solution
	$u \in H^1(\Omega)$ of~\eqref{eq:robin_problem} satisfies 
	\begin{enumerate}[(i)]
	\item\label{thm:robin_reg:noninv_lp}
		if $q < N/2$, then
		$u \in L^{Nq/(N-2q)}(\Omega)$, $u|_{\partial\Omega} \in L^{(N-1)q/(N-2q)}(\partial\Omega)$, and
		\begin{align*}
			& \|u\|_{L^{Nq/(N-2q)}(\Omega)} + \|u\|_{L^{(N-1)q/(N-2q)}(\partial\Omega)} \\
				& \quad \le c \bigl(\|u\|_{L^2(\Omega)} + \|f_0\|_{L^{q+\eps}(\Omega)}
					+ \sum_{j=1}^N \|f_j\|_{L^{Nq/(N-q)+\eps}(\Omega)} + \|g\|_{L^{(N-1)q/(N-q)+\eps}(\partial\Omega)}\bigr);
		\end{align*}
	\item\label{thm:robin_reg:noninv_C}
		if $q = N/2$, then
		$u \in \mathrm{C}^{0,\gamma}(\Omega)$, and
		\[
			\|u\|_{\mathrm{C}^{0,\gamma}(\Omega)}
				\le c \bigl(\|u\|_{L^2(\Omega)} + \|f_0\|_{L^{N/2+\eps}(\Omega)} + \sum_{j=1}^N \|f_j\|_{L^{N+\eps}(\Omega)} + \|g\|_{L^{N-1+\eps}(\partial\Omega)}\bigr).
		\]
	\end{enumerate}
	Moreover, if the solution is unique, then it satisfies
	\begin{enumerate}[(i)]\setcounter{enumi}{2}
	\item\label{thm:robin_reg:inv_lp}
		if $q < N/2$, then
		$u \in L^{Nq/(N-2q)}(\Omega)$, $u|_{\partial\Omega} \in L^{(N-1)q/(N-2q)}(\partial\Omega)$, and
		\begin{align*}
			& \|u\|_{L^{Nq/(N-2q)}(\Omega)} + \|u\|_{L^{(N-1)q/(N-2q)}(\partial\Omega)} \\
				& \quad \le c \bigl(\|f_0\|_{L^{q+\eps}(\Omega)}
					+ \sum_{j=1}^N \|f_j\|_{L^{Nq/(N-q)+\eps}(\Omega)} + \|g\|_{L^{(N-1)q/(N-q)+\eps}(\partial\Omega)}\bigr);
		\end{align*}
	\item\label{thm:robin_reg:inv_C}
		if $q = N/2$, then
		$u \in \mathrm{C}^{0,\gamma}(\Omega)$, and
		\[
			\|u\|_{\mathrm{C}^{0,\gamma}(\Omega)}
				\le c \bigl(\|f_0\|_{L^{N/2+\eps}(\Omega)} + \sum_{j=1}^N \|f_j\|_{L^{N+\eps}(\Omega)} + \|g\|_{L^{N-1+\eps}(\partial\Omega)}\bigr).
		\]
	\end{enumerate}
\end{theorem}
\begin{proof}
	We can consider solutions of~\eqref{eq:robin_problem} as solutions of a problem of the
	kind~\eqref{eq:robin_problem:omega} as in~\eqref{eq:robin_Lpreg:modeq},
	where $\omega$ is as in Lemma~\ref{lem:robin_laxmilgram} or
	Lemma~\ref{lem:robin_laxmilgram2}, respectively.
	For $N \ge 3$, Lemma~\ref{lem:robin_Lpreg} asserts that for this equation the
	assumptions of Lemma~\ref{lem:neumann_lpreg} ($q < N/2$) or Proposition~\ref{prop:neumann_Creg}
	($q = N/2$) are satisfied and that we can estimate the norm of
	$\|u\|_{L^{q+\eps}(\Omega)}$ appropriately, whereas for $N=2$ this is obvious.
	From this we obtain~\eqref{thm:robin_reg:noninv_lp} and~\eqref{thm:robin_reg:noninv_C}.

	For~\eqref{thm:robin_reg:inv_lp} and~\eqref{thm:robin_reg:inv_C} it only remains
	to show that $\|u\|_{L^2(\Omega)}$ can be estimated accordingly if the solution
	is unique. This can be proved using the Fredholm alternative,
	see Remark~\ref{rem:fredholm}.
\end{proof}

\section{Parabolic Problems}\label{sec:parabolic}
Again, let $\Omega \subset \mathds{R}^N$ be Lipschitz regular
and assume that $L$ is a strictly elliptic differential operator on
$\Omega$ as defined in Section~\ref{sec:prelim}.
It has been shown in~\cite{Warma06} that $-L = \Delta$
with Robin or Wentzell-Robin boundary conditions
generates a $\mathrm{C}_0$-semigroup
on $\mathrm{C}(\overline{\Omega})$, although the calculations
contain a small mistake that oversimplifies the arguments.
We employ similar ideas to show the result is true for general
operators.

\subsection{Neumann and Robin boundary conditions}
Let $A$ be the operator on $L^2(\Omega)$ associated with
the form $a_{L,\beta}$ defined in~\eqref{eq:robin_form}.
It follows from the theory of forms that $-A$ generates a positive, compact, holomorphic
$\mathrm{C}_0$-semigroup $(T(t))_{t \ge 0}$ on $L^2(\Omega)$,
cf.\ for example~\cite{Ouhabaz05}.
The trajectories of this semigroup are the unique mild solutions
of the parabolic problems~\eqref{eq:robin_parabolic}
with Robin boundary conditions, compare~\cite[\S VI.5]{EN00}.

It is known that each $T(t)$ is a kernel operator with a bounded kernel $k(t,\cdot,\cdot)$
which has Gaussian estimates~\cite[Corollary~6.1]{Daners00}.
Thus $(T(t))_{t \ge 0}$ extrapolates to a family of holomorphic semigroups on $L^p(\Omega)$,
$p \in [1,\infty]$, which have the same angle of holomorphy,
and all operator $T(z)$ for $\Real z > 0$
are kernel operators satisfying a Gaussian estimate~\cite[Theorem~5.4]{AtE97}.

We start this section by an investigation of the regularity of these kernels.
In particular it follows from the next theorem that the kernels are jointly
continuous in the time variable (away from $t=0$) and in the space variables
(up to the boundary of $\Omega$).
\begin{theorem}
	The function $t \mapsto k(t,\cdot,\cdot)$ is analytic from
	$(0,\infty)$ to $\mathrm{C}^{0,\gamma}(\Omega \times \Omega)$
	for $\gamma$ as in Theorem~\ref{thm:robin_reg}.
	In particular,
	$k \in \mathrm{C}^{0,\gamma}( [\tau_1,\tau_2] \times \Omega \times \Omega )$
	for $0 < \tau_1 \le \tau_2 < \infty$.
\end{theorem}
\begin{proof}
	Let $\omega$ be so large that $a_{L,\beta}^\omega$ is coercive.
	Then in particular $\lambda \coloneqq -\omega \in \rho(A)$.
	By Theorem~\ref{thm:robin_reg}
	there exists $m \in \mathds{N}$ and $\gamma > 0$ such that
	\[
		R(\lambda,A)^mL^2(\Omega) \subset \mathrm{C}^{0,\gamma}(\Omega).
	\]
	Since $(T(t))_{t \ge 0}$ is holomorphic, this implies that $T(t)$ maps
	$L^2(\Omega)$ boundedly to $\mathrm{C}^{0,\gamma}(\overline{\Omega})$ for every $t > 0$.

	Let $\phi_{\mathrm{hol}}$ be the sector of holomorphy of $(T(t))_{t \ge 0}$
	and fix $0 < \theta < \phi_{\mathrm{hol}}$.
	Let $k(z,\cdot,\cdot)$ denote the kernel of $T(z)$ for $z \in \Sigma_\theta$,
	and let $0 < \tau_1 < \tau_2$.
	Define
	\[
		\Sigma_{\theta,\tau_1,\tau_2} \coloneqq \{ z \in \mathds{C} : z - \tau_1 \in \Sigma_\theta \text{ and } |z| < \tau_2 \}.
	\]
	Since $k(t,\cdot,\cdot) \in L^\infty(\Omega \times \Omega)$,
	there exists a constant $K > 0$ that depends only on the semigroup
	and the set $\Sigma_{\theta,\tau_1,\tau_2}$ such that
	for all $z \in \Sigma_{\theta,\tau_1,\tau_2}$ and almost every
	$y \in \Omega$ we have
	\begin{align*}
		\left\| k(z,\cdot,y) \right\|_{\mathrm{C}^{0,\gamma}(\Omega)}
			& = \left\| T(\tfrac{\tau_1}{2}) T(z-\tau_1) k(\tfrac{\tau_1}{2},\cdot,y) \right\|_{\mathrm{C}^{0,\gamma}(\Omega)} \\
			& \le \left\| T(\tfrac{\tau_1}{2}) \right\|_{\mathscr{L}(L^2(\Omega), \mathrm{C}^{0,\gamma}(\Omega))}
				\left\| T(z-\tau_1) \right\|_{\mathscr{L}(L^2(\Omega))}
				\left\| k(\tfrac{\tau_1}{2},\cdot,y) \right\|_{L^2(\Omega)} \\
			& \le K
	\end{align*}
	Using a duality argument, we can estimate
	$\| k(z,x,\cdot) \|_{\mathrm{C}^{0,\gamma}(\Omega)}$ in a similar manner,
	possibly increasing the value of $K$ appropriately.
	Thus
	\begin{equation}\label{eq:sephoelder_jointhoelder}
		| k(z,x,y) - k(z,\bar{x},\bar{y}) |
			\le K | x - \bar{x} |^\gamma + K | y - \bar{y} |^\gamma
			\le 2K \left| \begin{smallpmatrix} x - \bar{x} \\ y - \bar{y} \end{smallpmatrix} \right|_\infty^\gamma
	\end{equation}
	for almost every $x$, $\bar{x}$, $y$ and $\bar{y}$ in $\Omega$,
	which shows that $\{ k(z,\cdot,\cdot) : z \in \Sigma_{\theta,\tau_1,\tau_2} \}$
	is a bounded subset of $\mathrm{C}^{0,\gamma}(\Omega \times \Omega)$.

	Since $(T(z))_{z \in \Sigma_\theta}$ is holomorphic on $L^2(\Omega)$,
	\[
		\scalar[L^2(\Omega)]{T(z)\setone_A}{\setone_B}
			= \int_{\Omega \times \Omega} k(z,x,y) \setone_{A \times B} \dx(x,y)
	\]
	is holomorphic for all measurable subsets $A$ and $B$ of $\Omega$.
	Considering integration against $\setone_{A \times B}$ as a functional
	on $\mathrm{C}^{0,\gamma}(\Omega \times \Omega)$, we see that the mapping
	$z \mapsto k(z,\cdot,\cdot)$ is holomorphic from $\Sigma_{\theta,\tau_1,\tau_2}$
	to $\mathrm{C}^{0,\gamma}(\Omega \times \Omega)$,
	because the linear combinations of such indicators separate the points of
	$\mathrm{C}^{0,\gamma}(\Omega \times \Omega)$,
	compare~\cite[Theorem~A.7]{ABHN01}.
	Since $\tau_1$ and $\tau_2$ are arbitrary, the first assertion follows.
	The second assertion is an easy consequence of the first.
\end{proof}

Next we show that $(T(t))_{t \ge 0}$ restricts to a 
$\mathrm{C}_0$-semigroup on $\mathrm{C}(\overline{\Omega})$.
For this we need the following density result.
\begin{lemma}\label{lem:robin_dense}
	Assume that $a_{L,\beta}$ is coercive and let $\gamma$ be as in
	Theorem~\ref{thm:robin_reg}.
	For all $v \in \mathrm{C}^\infty(\overline{\Omega})$ and all $\eps > 0$
	there exists $\psi \in \mathrm{C}^\infty(\overline{\Omega})$ such that
	the unique solution $u$ of~\eqref{eq:robin_problem} for the right-hand side
	$f_0 \coloneqq \psi$, $f_j \coloneqq 0$, $j=1,\dots,N$, and $g \coloneqq 0$
	satisfies $\|u-v\|_{\mathrm{C}^{0,\gamma}(\Omega)} < \eps$.
\end{lemma}
\begin{proof}
	Let $\tilde{\eps} > 0$ and $p > N$ be arbitrary, and let
	$h_d$ be in $\mathrm{C}^\infty(\mathds{R}^N;\mathds{R}^N)$ such that $h_d \cdot \nu \ge 1$
	almost everywhere on $\partial\Omega$. Such a vector field $h_d$ exists, see~\cite[Lemma~3.2]{Daners09}.
	By the Stone-Weierstrass theorem we can find a smooth vector field
	$h \in \mathrm{C}^\infty(\mathds{R}^N;\mathds{R}^N)$ such that
	\[
		\bigl\| h - \tfrac{\beta v h_d}{h_d \cdot \nu} \bigr\|_{L^p(\partial\Omega;\mathds{R}^N)} < \tilde{\eps}.
	\]
	Hence $\tilde{g} \coloneqq h \cdot \nu - \beta v$ satisfies
	$\| \tilde{g} \|_{L^p(\partial\Omega)} < \tilde{\eps}$.
	Since the test functions are dense in $L^p(\Omega)$, there exist
	$k_0$, $k_1$, \dots, $k_N$ in $\mathrm{C}^\infty_c(\Omega)$ such that
	the functions
	\[
		\tilde{f}_0 \coloneqq k_0 - \sum_{i=1}^N c_i D_iv - dv,
		\quad
		\tilde{f}_j \coloneqq -h_j - k_j - \sum_{i=1}^N a_{ij} D_iv - b_j v,
	\]
	$j=1,\dots,N$, satisfy $\|\tilde{f}_j\|_{L^p(\Omega)} < \tilde{\eps}$ for $j=0,\dots,N$.
	Define $\psi \in \mathrm{C}^\infty(\overline{\Omega})$ by
	\[
		\psi \coloneqq k_0 + \sum_{j=1}^N D_jk_j + \Div(h),
	\]
	and let $u$ be the unique solution of~\eqref{eq:robin_problem}
	as described in the claim. By the divergence theorem,
	\begin{align*}
		\int_\Omega \psi \phi \; \dx\lambda
			= \int_\Omega k_0 \phi \; \dx\lambda
				+ \int_{\partial\Omega} (h \cdot \nu) \phi \; \dx\sigma
				- \sum_{j=1}^N \int_\Omega (h_j + k_j) D_j\phi \; \dx\lambda
	\end{align*}
	for every $\phi \in \mathrm{C}^1(\overline{\Omega})$, hence
	\begin{align*}
		a_{L,\beta}(u - v, \phi)
			& = \int_\Omega \psi \phi \; \dx\lambda - a_{L,\beta}(v,\phi) \\
			& = \int_\Omega \tilde{f}_0 \phi \; \dx\lambda + \sum_{j=1}^N \int_\Omega \tilde{f}_j D_j\phi \; \dx\lambda + \int_{\partial\Omega} \tilde{g}\phi \; \dx\sigma
	\end{align*}
	for all $\phi \in \mathrm{C}^1(\overline{\Omega})$.
	Thus part~\eqref{thm:robin_reg:inv_C} of Theorem~\ref{thm:robin_reg} implies that
	\[
		\| u - v \|_{\mathrm{C}^{0,\gamma}(\Omega)}
			\le c (N+2) \tilde{\eps}
	\]
	for some constant $c$ that does not depend on $v$, $\psi$, $\eps$ or $\tilde{\eps}$.
	Now if we pick $\tilde{\eps}$ small enough such that $c (N+2) \tilde{\eps} < \eps$,
	the claim follows.
\end{proof}

\begin{theorem}\label{thm:robin_semigroup}
	The restriction of $(T(t))_{t \ge 0}$ to $\mathrm{C}(\overline{\Omega})$
	is a positive, compact, holomorphic $\mathrm{C}_0$-semigroup.
\end{theorem}
\begin{proof}
	Let $\omega$ be such that $a_{L,\beta}^\omega$ is coercive.
	Then in particular $\lambda \coloneqq -\omega \in \rho(A)$.
	By Theorem~\ref{thm:robin_reg}
	there exists $m \in \mathds{N}$ and $\gamma > 0$ such that
	\[
		R(\lambda,A)^mL^2(\Omega) \subset \mathrm{C}^{0,\gamma}(\Omega).
	\]
	Since $(T(t))_{t \ge 0}$ is holomorphic, this implies that $T(t)$ maps
	$L^2(\Omega)$ boundedly to $\mathrm{C}^{0,\gamma}(\overline{\Omega})$ for every $t \ge 0$.
	In particular, the subspace $\mathrm{C}(\overline{\Omega})$ is invariant
	under $T(t)$, and factoring through $L^2(\Omega)$ we see that $T(t)$ is
	a compact operator on $\mathrm{C}(\overline{\Omega})$.
	Positivity follows from the positivity on $L^2(\Omega)$.

	As was already remarked,
	the restriction of $(T(t))_{t \ge 0}$ to $L^\infty(\Omega)$ is
	a holomorphic semigroup in the sense of~\cite[Definition~3.7.1]{ABHN01}.
	Its generator is the part of $A$ in $L^\infty(\Omega)$.
	Since $\mathrm{C}^\infty(\overline{\Omega})$ is dense in $\mathrm{C}(\overline{\Omega})$,
	Lemma~\ref{lem:robin_dense} shows that the part of $A + \omega$ in $\mathrm{C}(\overline{\Omega})$
	and hence also the part of $A$ in $\mathrm{C}(\overline{\Omega})$ is densely defined.
	Thus, by~\cite[Proposition~3.7.4 and Remark~3.7.13]{ABHN01}, the restriction of
	$(T(t))_{t \ge 0}$ to $\mathrm{C}(\overline{\Omega})$ is a holomorphic
	$\mathrm{C}_0$-semigroup, whose generator is the part of $A$ in $\mathrm{C}(\overline{\Omega})$.
\end{proof}

\subsection{Wentzell-Robin boundary conditions}
Let $A$ be the operator on the Hilbert space $\mathcal{H} \coloneqq L^2(\Omega) \oplus L^2(\partial\Omega)$
that is associated with the form
\[
	\mathfrak{a}_{L,\beta}( (u,u|_{\partial\Omega}), (v,v|_{\partial\Omega}) )
		\coloneqq a_{L,\beta}(u,v)
\]
with the dense form domain
\[
	\mathcal{V} \coloneqq \{ (u, u|_{\partial\Omega}) : u \in H^1(\Omega) \}
		\subset \mathcal{H}.
\]
It follows from the theory of forms that $-A$ generates
a positive, compact, holomorphic $\mathrm{C}_0$-semigroup $(T(t))_{t \ge 0}$ on $\mathcal{H}$.
This semigroup, or more precisely its restriction to $\mathcal{V}$, describes
the solutions of the evolution problem~\eqref{eq:wentzell_parabolic} with Wentzell-Robin
boundary conditions, compare~\cite{AMPR03}.

We need to show that $(T(t))_{t \ge 0}$ extrapolates to $\mathrm{C}(\overline{\Omega})$.
An easy sufficient condition is quasi-$L^\infty$-contractivity,
i.e., to assume that the semigroup $(\e^{-\omega t}T(t))_{t \ge 0}$
is $L^\infty$-contractive for some $\omega \in \mathds{R}$.
However, this cannot be expected in general, even if $\Omega$ is an
interval and $L$ is formally self-adjoint and has regular second-order
coefficients, as the following example shows.

\begin{example}\label{ex:nonsubmark}
	Consider the operator
	\[
		(Lu)(x) = -(u'(x) + \sgn(x)u(x))' + \sgn(x)u(x)
	\]
	on $\Omega = (-1,1)$ with Wentzell-Robin boundary conditions,
	i.e., $a = 1$, $b = c = \sgn$, $d=0$, and $\beta$ arbitrary.
	There exists no $\omega \in \mathds{R}$ such that the semigroup
	$\e^{-\omega t} T(t)$ consists of contractions on
	$L^\infty(\Omega) \oplus L^\infty(\partial\Omega)$.
\end{example}
\begin{proof}
	Assume that $\e^{-\omega t}T(t)$ is contractive on
	$L^\infty(\Omega) \oplus L^\infty(\partial\Omega)$ for all $t \ge 0$.
	This semigroup comes from the form $\mathfrak{a}_{L,\beta}^\omega$,
	which is defined by
	\begin{align*}
		& \mathfrak{a}_{L,\beta}^\omega( (u,u|_{\partial\Omega}), (v,v|_{\partial\Omega}) ) \\
			& \qquad \coloneqq \int_{-1}^1 \bigl(u'(x) v'(x) + \sgn(x) u(x) v'(x) + \sgn(x) u'(x) v(x) + \omega u(x) v(x)\bigr) \dx x \\
				& \qquad\qquad + \bigl(\beta(-1) + \omega\bigr) u(-1) v(-1) + \bigl(\beta(1) + \omega\bigr) u(1)v(1)
	\end{align*}
	for $(u,u|_{\partial\Omega})$ and $(v,v|_{\partial\Omega})$ in $\mathcal{V}$.
	Assume that there exists $\omega \in \mathds{R}$ such that $\e^{-\omega t}T(t)$ is
	$L^\infty$-contractive for all $t \ge 0$.
	By~\cite[Theorem~2.15]{Ouhabaz05} this implies that
	\[
		\mathfrak{a}_{L,\beta}^\omega\bigl( (v, v|_{\partial\Omega}), (w, w|_{\partial\Omega}) \bigr) \ge 0
		\text{ with } v \coloneqq (1 \wedge |u|) \sgn(u) \text{ and } w \coloneqq (|u|-1)^+ \sgn(u)
	\]
	for all $(u,u|_{\partial\Omega}) \in \mathcal{V}$, hence in particular
	\[
		\int_{-1}^1 \bigl(\sgn(x) u'(x) + \omega u(x)\bigr) \setone_{\{u \ge 1\}} \dx x
			\ge 0
	\]
	for all $u \in H^1(-1,1)$ satisfying $u(-1) = u(1) = 0$ and $u \ge 0$.
	For $u_n(x) \coloneqq 2(1-x^2)^n$ we thus obtain for $\alpha_n \coloneqq (1 - 2^{-1/n})^{1/2}$
	\begin{align*}
		0 & \le \int_{-1}^1 \bigl(\sgn(x) u_n'(x) + \omega u_n(x)\bigr) \setone_{\{u_n \ge 1\}} \dx x \\
			& = -u_n\bigr|^0_{-\alpha_n} + u_n\bigr|^{\alpha_n}_0 + \omega \int_{-\alpha_n}^{\alpha_n} u_n \dx\lambda
			\le -2 + 4\omega \alpha_n.
	\end{align*}
	This is a contradiction since $\alpha_n \to 0$ as $n \to \infty$.
\end{proof}

However, if we assume some regularity of the coefficients
we obtain a quasi-submarkovian semigroup, i.e., a semigroup
such that $(\e^{-\omega t}T(t))_{t \ge 0}$ is positive and $L^\infty$-contractive
for some $\omega \in \mathds{R}$, as we show next.

\begin{proposition}\label{prop:submarkovian}.
	If $b_j \in W^{1,\infty}(\Omega)$ for $j=1,\dots,N$,
	then the Wentzell-Robin semigroup $(T(t))_{t \ge 0}$ is quasi-submarkovian
	on $\mathcal{H}$.
\end{proposition}
\begin{proof}\allowdisplaybreaks
	It follows from~\cite[Theorem~2.6]{Ouhabaz05} that $(T(t))_{t \ge 0}$ is positive.
	By assumption, there exists $k \ge 0$ such that $|b| \le k$
	and $\Div(b) \le k$ almost everywhere. Pick $\omega$ larger than 
	$\|d\|_\infty + k$ and $\|\beta\|_\infty + k$.
	Since $D_j u^+ = D_ju \setone_{\{u \ge 0\}}$, we obtain from the divergence theorem
	that for all $(u,u_{\partial\Omega}) \in \mathcal{V}$ satisfying $u \ge 0$
	we have
	\begin{align*}
		& \mathfrak{a}_{L,\beta}^\omega\bigl( (1 \wedge u, 1 \wedge u|_{\partial\Omega}), ( (u-1)^+, (u|_{\partial\Omega}-1)^+) \bigr) \\
			& = \int_\Omega b \, \nabla (u-1)^+ \; \dx\lambda + \int_\Omega (d+\omega) (u-1)^+ \; \dx\lambda + \int_{\partial\Omega} (\beta + \omega) (u-1)^+ \; \dx\sigma \\
			& = \int_{\partial\Omega} (u-1)^+ \, b \cdot \nu \; \dx\sigma - \int_\Omega (u-1)^+ \Div(b) \; \dx\lambda \\
			& \qquad + \int_\Omega (d+\omega) (u-1)^+ \; \dx\lambda  + \int_{\partial\Omega} (\beta + \omega) (u-1)^+ \; \dx\sigma \\
			& \ge \int_\Omega (u-1)^+ (\omega-\|d\|_\infty-k) \; \dx\lambda + \int_{\partial\Omega} (u-1)^+ (\omega-\|\beta\|_\infty-k) \; \dx\sigma
			\ge 0.
	\end{align*}
	It follows from~\cite[Corollary~2.17]{Ouhabaz05} that $(\e^{-\omega t}T(t))_{t \ge 0}$
	is submarkovian.
\end{proof}

We need a density result, which is similar to Lemma~\ref{lem:robin_dense},
in order to show that $(T(t))_{t \ge 0}$ restricts to a $\mathrm{C}_0$-semigroup
on $\mathrm{C}(\overline{\Omega})$.
\begin{lemma}\label{lem:wentzell_dense}
	Assume that $a_{L,\beta}$ is coercive and let $\gamma$ be as in Theorem~\ref{thm:robin_reg}.
	For all $v \in \mathrm{C}^\infty(\overline{\Omega})$ and all $\eps > 0$
	there exists $\psi \in \mathrm{C}^\infty(\overline{\Omega})$ such that
	the unique solution $u$ of~\eqref{eq:robin_problem} for the right-hand side
	$f_0 \coloneqq \psi$, $f_j \coloneqq 0$, $j=1,\dots,N$, and $g \coloneqq \psi|_{\partial\Omega}$
	satisfies $\|u-v\|_{\mathrm{C}^{0,\gamma}(\Omega)} < \eps$.
\end{lemma}
\begin{proof}
	Let $p > N$ and $\tilde{\eps} > 0$ be arbitrary.
	By the Stone-Weierstrass theorem there exists
	$\tilde{k}_0 \in \mathrm{C}^\infty(\overline{\Omega})$ such that
	$\tilde{g} \coloneqq (\tilde{k}_0 - \beta v)|_{\partial\Omega}$
	satisfies $\|\tilde{g}\|_{L^p(\partial\Omega)} < \tilde{\eps}$.
	Now pick test functions $k_j \in \mathrm{C}^\infty_c(\Omega)$, $j=1,\dots,N$
	such that
	\[
		\tilde{f}_0 \coloneqq \tilde{k}_0 + k_0 - \sum_{i=1}^N c_i D_iv - dv,
		\quad
		\tilde{f}_j \coloneqq k_j - \sum_{i=1}^N a_i D_i v - b_j v
	\]
	satisfy $\|\tilde{f}_j\|_{L^p(\Omega)} < \tilde{\eps}$ for $j=0,\dots,N$.
	Define
	\[
		\psi \coloneqq \tilde{k}_0 + k_0 + \sum_{j=1}^N k_j \in \mathrm{C}^\infty(\overline{\Omega})
	\]
	and let $u$ be the unique solution of~\eqref{eq:robin_problem} as described
	in the claim. Then
	\begin{align*}
		a_{L,\beta}(u-v, \phi)
			& = \int_\Omega \psi \phi \; \dx\lambda + \int_{\partial\Omega} \psi \phi \; \dx\sigma
				- a_{L,\beta}(v,\phi) \\
			& = \int_\Omega \tilde{f}_0 \phi \; \dx\lambda + \sum_{j=1}^N \int_\Omega \tilde{f}_j D_j\phi \; \dx\lambda
				+ \int_{\partial\Omega} \tilde{g} \phi \; \dx\sigma
	\end{align*}
	for all $\phi \in \mathrm{C}^1(\overline{\Omega})$.
	Thus part~\eqref{thm:robin_reg:inv_C} of Theorem~\ref{thm:robin_reg} implies that
	\[
		\|u-v\|_{\mathrm{C}^{0,\gamma}(\Omega)} < c(N+2)\tilde{\eps}
	\]
	for some constant $c$ that does not depend on $v$, $\psi$, $\eps$ or $\tilde{\eps}$.
	If we pick $\tilde{\eps}$ small enough such that $c(N+2)\tilde{\eps} < \eps$,
	the claim follows.
\end{proof}

\begin{theorem}\label{thm:wentzell_semigroup}
	Assume $b_j \in W^{1,\infty}(\Omega)$ for all $j=1,\dots,N$.
	Then the restriction of $(T(t))_{t \ge 0}$ to
	$\mathscr{C} \coloneqq \{ (u,u|_{\partial\Omega}) : u \in \mathrm{C}(\overline{\Omega}) \}$
	is a positive, compact $\mathrm{C}_0$-semigroup.
\end{theorem}
\begin{proof}
	Pick $\omega \ge 0$ large enough such that $a_{L,\beta}^\omega$ and hence in
	particular $\mathfrak{a}_{L,\beta}^\omega$ is coercive.
	Then $\lambda \coloneqq -\omega$ is in $\rho(A)$, where $A$ denotes the generator
	of $(T(t))_{t \ge 0}$.

	Using Theorem~\ref{thm:robin_reg}, one can show as in the proof of
	Lemma~\ref{lem:robin_Lpreg} that there exists $m \in \mathds{N}$ such that
	\[
		R(\lambda,A)^m \mathcal{H}
			\subset \mathscr{C}^{0,\gamma} \coloneqq \left\{ (u,u|_{\partial\Omega}) : u \in \mathrm{C}^{0,\gamma}(\Omega) \right\}.
	\]
	Since $(T(t))_{t \ge 0}$ is analytic, each $T(t)$, $t > 0$, is a bounded operator
	from $\mathcal{H}$ to $\mathscr{C}^{0,\gamma}$.
	In particular, $\mathscr{C}$ is invariant under each $T(t)$, $t > 0$.

	By Proposition~\ref{prop:submarkovian}, the restriction of
	$(T(t))_{t \ge 0}$ to $\mathscr{C}$
	is a semigroup in the sense of~\cite[Definition~3.2.5]{ABHN01}.
	Its generator is the part of $A$ in $\mathscr{C}$.
	Since $\mathrm{C}^\infty(\overline{\Omega})$ is dense in $\mathrm{C}(\overline{\Omega})$,
	the part of $A$ in $\mathscr{C}$ is densely defined by Lemma~\ref{lem:wentzell_dense}.
	Thus, by~\cite[Corollary~3.3.11]{ABHN01}, the restriction of $(T(t))_{t \ge 0}$
	to $\mathscr{C}$ is a $\mathrm{C}_0$-semigroup.

	Since $T(t)$ is positive on $\mathcal{H}$, it is also positive on $\mathscr{C}$.
	Since $\mathcal{V}$ is compactly embedded into $\mathcal{H}$ by the Sobolev
	embedding theorems, $T(t)$ is compact on $\mathcal{H}$ for every $t \ge 0$.
	Compactness of the semigroup on $\mathscr{C}$ follows by factorization through $\mathcal{H}$.
\end{proof}

\begin{remark}
	Typically one identifies the semigroup $(T(t))_{t \ge 0}$ on $\mathscr{C}$
	of Theorem~\ref{thm:wentzell_semigroup}
	via the isometric isomorphism
	\[
		\mathscr{C} \to \mathrm{C}(\overline{\Omega}), \; (u,u|_{\partial\Omega}) \mapsto u
	\]
	with a positive, compact $\mathrm{C}_0$-semigroup on $\mathrm{C}(\overline{\Omega})$
	and calls that one the Wentzell-Robin semigroup.
\end{remark}

In the proof of the preceding theorem we used the regularity assumption
on the coefficients only to ensure that the operator norm of $T(t)$
on $L^\infty(\Omega) \oplus L^\infty(\partial\Omega)$ is bounded for small $t$.
There seems to be no simple argument that assures the boundedness
in the general case. For example, as we have seen in Example~\ref{ex:nonsubmark},
we cannot expect the semigroup to be quasi-contractive.

However, the situation is different for the $L^p$-spaces, $1 < p < \infty$.
By direct estimates, Daners~\cite{Daners00} proved that under rather general regularity assumptions
that the Robin semigroup is quasi-$L^p$-contractive for every $p \in (1,\infty)$.
Although a similar proof still works for the Wentzell-Robin semigroup in the product space
$L^p(\Omega) \oplus L^p(\partial\Omega)$, as the last result in this article
we show how such an estimate can be obtained by reduction to the Robin case,
which extends the result in~\cite{FGGR08}.

\begin{proposition}\label{prop:lp_contr}
	There exists a $\delta_0$ that depends only on the coefficients of the
	differential operator $L$ such that for every $p \in (1,\infty)$ we have
	\[
		\|T(t)u\|_p \le \e^{\omega_p t} \|u\|_p
	\]
	for all $t \ge 0$ and all $u \in \mathcal{H}$ that are in
	$L^p(\Omega) \oplus L^p(\partial\Omega)$. Here we write
	$\omega_p$ for the quantity $\max\{p,p'\} \delta_0$, where
	$p'$ denotes the dual exponent to $p$.
\end{proposition}
\begin{proof}
	Let $p \in (1,\infty)$.
	Denote by $\mathcal{B}$ the intersection of $\mathcal{H}$ and the unit ball of
	$L^p \coloneqq L^p(\Omega) \oplus L^p(\partial\Omega)$.
	By Fatou's lemma, $\mathcal{B}$ is closed in $\mathcal{H}$, and
	$\mathcal{B}$ is convex.
	Let $\mathcal{P}$ denote the orthogonal projection of $\mathcal{H}$ onto $\mathcal{B}$.

	Since in the special case $L = -\Delta$ and $\beta = 0$
	the corresponding semigroup $(S(t))_{t \ge 0}$ is quasi-submarkovian by
	Proposition~\ref{prop:submarkovian}, we obtain from the
	Riesz-Thorin interpolation theorem that there exists
	$\omega > 0$ such that $(\e^{-\omega t}S(t))_{t \ge 0}$
	leaves $\mathcal{B}$ invariant. Thus
	\begin{equation}\label{eq:lp_contr:form_domain}
		\mathcal{P}\mathcal{V} \subset \mathcal{V}
	\end{equation}
	by~\cite[Theorem~2.2]{Ouhabaz05}, since $\mathcal{V}$
	is the form domain of $\mathfrak{a}_{-\Delta,0}^\omega$.

	Let $(R(t))_{t \ge 0}$ denote the Robin semigroup for the form $a_{L,\beta}$.
	Let $B$ denote the intersection of $L^2(\Omega)$ and the closed unit ball in $L^p(\Omega)$,
	and let $P$ be the orthogonal projection of $L^2(\Omega)$ onto $B$.
	By~\cite[Theorem~5.1]{Daners00}, $(\e^{-\omega_p t}R(t))_{t \ge 0}$ maps $B$ into itself.
	Thus
	\begin{equation}\label{eq:lp_contr:robin_form}
		a_{L,\beta}^{\omega_p}(u,u-Pu) \ge 0
		\text{ for all } u \in V \coloneqq H^1(\Omega)
	\end{equation}
	by~\cite[Theorem~2.2]{Ouhabaz05}.
	Since we already know that $\mathcal{P}$ maps $\mathcal{V}$ into $\mathcal{V}$,
	it is easy to see that
	\begin{equation}\label{eq:lp_contr:proj_rel}
		\mathcal{P}(u, u|_{\partial\Omega}) = (Pu, (Pu)|_{\partial\Omega})
		\text{ for all } u \in H^1(\Omega).
	\end{equation}
	By definition of $\mathfrak{a}_{L,\beta}^{\omega_p}$, it follows
	from~\eqref{eq:lp_contr:robin_form} and~\eqref{eq:lp_contr:proj_rel} that
	\begin{equation}\label{eq:lp_contr:wentzell_form}
		\mathfrak{a}_{L,\beta}^{\omega_p}\bigl( (u,u|_{\partial\Omega}), (I-\mathcal{P})(u, u|_{\partial\Omega}) \bigr)
			\ge 0 \text{ for all } u \in H^1(\Omega).
	\end{equation}

	Again by Theorem~\cite[Theorem~2.2]{Ouhabaz05} it follows
	from~\eqref{eq:lp_contr:form_domain} and~\eqref{eq:lp_contr:wentzell_form}
	that $\mathcal{B}$ is invariant under the semigroup $(\e^{-\omega_p t}T(t))_{t \ge 0}$.
	This is precisely the statement we wanted to prove.
\end{proof}


\bibliographystyle{amsalpha}
\bibliography{boundarycont}

\end{document}